%% file: Parabolic_DG_003.tex
\documentclass[a4paper, 11pt, american]{amsart}

\usepackage[colorlinks,pdfpagelabels,pdfstartview = FitH,bookmarksopen
= true,bookmarksnumbered = true,linkcolor = blue,plainpages =
false,hypertexnames = false,citecolor = red,pagebackref=false]{hyperref}
\usepackage{times,latexsym,amssymb}
\usepackage{amsmath,amsthm,bm}
\usepackage{graphics, epsfig}
\usepackage{color}
\usepackage{appendix}
\usepackage[margin=1.5in]{geometry}

\usepackage{amsbsy}
\usepackage{amstext}
\usepackage{amssymb}
\usepackage{esint}
\usepackage{stmaryrd}

\let\TeXchi\chi
\newbox\chibox
\setbox0 \hbox{\mathsurround0pt $\TeXchi$}
\setbox\chibox \hbox{\raise\dp0 \box 0 }
\def\chi{\copy\chibox}

\renewcommand{\d}{\mathrm{d}}
\newcommand{\dx}{\mathrm{d}x}

\newcommand{\dt}{\mathrm{d}t}

\renewcommand{\epsilon}{\varepsilon}
\renewcommand{\rho}{\varrho}

\author[N. Liao]{Naian Liao}
\address{Naian Liao\\
Fachbereich Mathematik, Universit\"at Salzburg\\
Hellbrunner Str. 34, 5020 Salzburg, Austria}
\email{naian.liao@sbg.ac.at}
\begin{document}
\input dibe_1009.tex

\title{Remarks on Parabolic De Giorgi Classes}
\date{}
\maketitle
\begin{abstract}
We make several remarks concerning properties of functions in parabolic De Giorgi classes of order $p$.
There are new perspectives including a novel mechanism of propagating positivity in measure,
the reservation of membership under convex composition, and a logarithmic type estimate.
Based on them, we are able to give new proofs of known properties.
In particular, we prove local boundedness and local H\"older continuity of these functions
 via Moser's ideas, thus avoiding De Giorgi's heavy machinery.
We also seize this opportunity to give a transparent
proof of a weak Harnack inequality for non-negative members of some super-class of De Giorgi,
without any covering argument.
\end{abstract}

\section{Introduction}
De Giorgi classes consist of Sobolev functions in an open set 
$\Om\subset \rr^N$ satisfying a family of energy estimates, i.e.
$u\in W^{1,p}_{\loc}(\Om)$ and for some $\gm>0$,
\begin{equation*}
\int_{K_{\rho}(y)}|D(u-k)_{\pm}|^p\,\dx\le\frac{\gm}{(R-\rho)^p}\int_{K_{R}(y)}|(u-k)_{\pm}|^p\,\dx
\end{equation*}
 for all $k\in\rr$, and any pair of concentric cubes $K_{\rho}(y)\Subset K_R(y)$ in $\Om$.
The significance of De Giorgi classes lies in that they are general enough
to include not only weak solutions to quasi-linear elliptic equation in divergence form (cf. \cite{DB, LU}), but also
local minima or quasi-minima of functionals that do not necessarily
admit any Euler equations (cf. \cite{GG}). 
Formulated by Ladyzhenskaya and Ural'tseva (cf. \cite{LU}), it has been shown that
functions in such classes (of elliptic nature) are locally H\"older continuous, using the beautiful ideas of De Giorgi
in his celebrated work \cite{DG}. A probably even more striking discovery was made
by DiBenedetto and Trudinger in \cite{DT} that non-negative members of De Giorgi classes
actually satisfy Harnack's inequality, which is a typical property of harmonic functions.
In addition to De Giorgi's techniques, 
the main new input of \cite{DT} includes realization of pointwise lower bound of non-negative members
in De Giorgi classes with
a power-like dependence on the measure distribution of their positivity.
The proof uses a deep covering
lemma due to Krylov and Safonov in \cite{KS}.

The original consideration by De Giorgi in \cite{DG} was to obtain H\"older continuity
of weak solutions to linear elliptic equations in divergence form with bounded and measurable
coefficients. Later on, Moser invented a new approach in \cite{Moser1} to show the same kind of result.
Moreover, he was able to obtain Harnack's inequality for such equations in \cite{Moser2}.
A key idea of Moser's new proof in \cite{Moser1} is to show a certain logarithmic function of the solution
is in fact a sub-solution and to formulate its energy estimates.
The feature of Moser's approach is twofold: on the one hand, it simplifies the
original proof of De Giorgi and gives a more intuitive method; on the other hand, it keeps referring to the equation.
This latter point renders a question on whether we could use Moser's idea in \cite{Moser1} to show the H\"older regularity
for functions in De Giorgi classes, where no equations are at our disposal.
Recently, an affirmative answer has been given in \cite{KL} based on a result in \cite{DBG}.
Naturally, one wonders if Moser's idea in \cite{Moser2} could be used to establish Harnack's
inequality for non-negative members of De Giorgi classes.
This, however, remains elusive.

A parabolic version of De Giorgi classes has been introduced in \cite{LSU}.
It should also be pointed out that different notions of parabolic De Giorgi
classes have been introduced in the literature. See for instance \cite{GV, Kinn, Lieberman}.
H\"older regularity has been established in \cite{LSU} employing De Giorgi's ideas.
Harnack's inequality is first established in \cite{W} using the covering lemma
of Krylov and Safonov. As in \cite{DT}, a weak Harnack inequality was proved in \cite{W},
which is of interest in its own right.
A direct proof of Harnack's inequality is presented in \cite{GV}, thus by-passing
a weak Harnack inequality.

The main goals of this note are the following.
In Section~\ref{S:5}, we give a proof of H\"older regularity for members of certain parabolic
De Giorgi classes, via Moser's ideas, thus avoiding De Giorgi's heavy machinery.
This parallels the result for the elliptic De Giorgi classes in \cite{KL}.
In Section~\ref{S:6}, we seize this opportunity to give a transparent proof of 
a weak Harnack inequality for non-negative members of a certain parabolic super-class of 
De Giorgi. The main tool is a measure theoretical lemma
established in \cite{DBGV}, thus by-passing the heavy covering argument of Krylov
and Safonov.
Last but not least, we show in Section~\ref{S:2} that local boundedness of functions
in parabolic De Giorgi classes can be achieved via Moser's iteration.
A similar observation has been made in \cite{DT} for the elliptic case.
In Section~\ref{S:3}, we show convex, non-decreasing functions of members in sub-classes of De Giorgi
are still in the same classes.
In Section~\ref{S:4}, we present some observation of the time propagation of 
measure information.
\subsection{Notations and Definitions}
Let $E$ be an open set in $\rr^{N}\times\rr$ and $(y,s)\in E$. Let $K_{\rho}(y)$ be a cube of edge $2\rho$
centered at $y\in\rr^N$. When $y=0$ we simply write $K_{\rho}$.
A cylinder with vertex at $(y,s)$, the base cube $K_\rho(y)$ and the length $\tau$ is defined by
\[
(y,s)+Q_{\rho,\tau}=K_{\rho}(y)\times (s-\tau, s].
\]
When $\tau=\rho^p$ for some $p>1$, we write $(y,s)+Q_\rho=K_{\rho}(y)\times (s-\rho^p,s]$.
When $(y,s)=(0,0)$, we omit it from the notation.

Suppose $u$ is a measurable function defined in $E$, such that for some $p>1$,
\[
 u\in L^{\infty}\left(s-T,s;L^p\big(K_R(y)\big)\right)\cap L^{p}\left(s-T,s;W^{1,p}\big(K_R(y)\big)\right)
\]
for any $(y,s)+Q_{R,T}\Subset E$.
We say $u$ belongs to the parabolic De Giorgi class $\mathfrak{A}^{\pm}_p(E,\gm)$ of order $p$, if
there exists a constant $\gm>0$ such that, for any $0<\rho<R$, $0<\tau<T$ and $k\in\rr$,
the following integral inequalities hold:
\begin{equation}\label{Eq:1:1}
\begin{aligned}
\essup_{s-\tau<t<s}&\int_{K_{\rho}(y)}(u-k)_{\pm}^p(\cdot,t)\,\dx
+\iint_{(y,s)+Q_{\rho,\tau}}|D(u-k)_{\pm}|^p\,\dx\dt\\
&\le\bigg[\frac{\gm}{(R-\rho)^p}+\frac{\gm}{T-\tau}\bigg]\iint_{(y,s)+Q_{R,T}}(u-k)_{\pm}^p\,\dx\dt.
\end{aligned}
\end{equation}
We define also the class $\mathfrak{A}_p(E,\gm):=\mathfrak{A}^{+}_p(E,\gm)\cap\mathfrak{A}^{-}_p(E,\gm)$.

Now suppose
\[
 u\in C\left(s-T,s;L^p\big(K_R(y)\big)\right)\cap L^{p}\left(s-T,s;W^{1,p}\big(K_R(y)\big)\right)
\]
We say $u$ belongs to the parabolic De Giorgi class $\mathfrak{B}^{\pm}_p(E,\gm)$
of order $p$, if
$u\in\mathfrak{A}^{\pm}_p(E,\gm)$ and in addition
the following integral inequalities hold for any $0<\rho<R$, $0<\tau<T$ and $k\in\rr$:
\begin{equation}\label{Eq:1:2}
\begin{aligned}
\essup_{s-T<t<s}\int_{K_{\rho}(y)}(u-k)_{\pm}^p(\cdot,t)\,\dx&
\le\int_{K_{R}(y)}(u-k)_{\pm}^p(\cdot,s-T)\,\dx\\
&\quad+\frac{\gm}{(R-\rho)^p}\iint_{(y,s)+Q_{R,T}}(u-k)_{\pm}^p\,\dx\dt.
\end{aligned}
\end{equation}
Analogously we define the class
 $\mathfrak{B}_p(E,\gm):=\mathfrak{B}^{+}_p(E,\gm)\cap\mathfrak{B}^{-}_p(E,\gm)$.

We remark that our definitions of De Giorgi classes mainly follow those in \cite{LSU}.
One difference is that we consider an arbitrary order $p>1$, whereas $p=2$ in \cite{LSU}. 
Also, a certain non-homogeneous term is imposed in \cite{LSU} for the inequalities \eqref{Eq:1:1}
and \eqref{Eq:1:2}. However, we decide to omit such a term for simplicity of presentation.

In the sequel, we refer to the set of parameters $\{\gm,\, p,\,N\}$ as the {\it data}
and use $C$ as a generic constant that can be quantitatively determined a priori
only in terms of the data.

Here and in the sequel, we will use $\mathcal{A}(R,T, \rho, \tau)$ to denote a generic positive, homogeneous quantity 
in the sense that under the relation $\rho=\sig_1R$, $\tau=\sig_2 T$ and $T=R^p$, it becomes a 
quantity of $\sig_1$ and $\sig_2$, possibly also depending on the data. 
We will say $u$ belongs to the generalized class $\mathfrak{A}^{\pm}_p$,
if \eqref{Eq:1:1} holds with $\gm$ replaced by $\mathcal{A}$. Similar definition holds for $\mathfrak{B}^{\pm}_p$.

\

\section{Local Boundedness of Functions in $\mathfrak{A}_p^{\pm}$}\label{S:2}
In general the membership in $\mathfrak{A}_p^{\pm}(E,\gm)$ does not
guarantee continuity. A Heaviside function of the time variable would be an example.
Nevertheless every function in $\mathfrak{A}_p^{\pm}(E,\gm)$ is locally bounded from above or from below.

\begin{theorem}\label{Thm:Bd}
Suppose $u\in\mathfrak{A}_p^{\pm}(E,\gm)$. 
Then there is a homogeneous quantity $\mathcal{A}$, such that
\begin{equation}
\essup_{(y,s)+Q_{\rho,\tau}}(u-k)_{\pm}\le\frac{\mathcal{A}}{(R-\rho)^{N}(T-\tau)}\iint_{(y,s)+Q_{R,T}}(u-k)_{\pm}\,\dx\dt
\end{equation}
for any cube $(y,s)+Q_{R,T}\subset E$ and all $k\in\rr$.
The same conclusion holds for members in the generalized classes $\mathfrak{A}^{\pm}_p$.
\end{theorem}
The proof is usually written using De Giorgi's iteration (cf. \cite{LSU, Lieberman}).
Nevertheless, we present here a proof based on Moser's iteration.
\subsection{Proof by Moser's Iteration}
Multiply both sides of $\eqref{Eq:1:1}_+$ by $k^{\be}$ with $\be>-1$ and integrate
in $\d k$ from $0$ to $\infty$ to get
\begin{align*}
\essup_{-\tau<t<0}&\int_0^{\infty}k^{\be}\,\d k\int_{K_{\rho}}(u(\cdot,t)-k)_{+}^p\,\dx+
\int_0^{\infty}k^{\be}\,\d k\iint_{Q_{\rho,\tau}}|D(u-k)_{+}|^p\,\dx\dt\\
&\le\bigg[\frac{\gm}{(R-\rho)^p}+\frac{\gm}{T-\tau}\bigg]\int_0^{\infty} k^{\be}\,\d k\iint_{Q_{R,T}}(u-k)_{+}^p\,\dx\dt.
\end{align*}
Fixing $-\tau<t<0$ and applying Fubini's theorem, the first term on the left-hand side is estimated by
\[
\int_{K_{\rho}}\int_0^{u}k^{\be}(u-k)_{+}^p\,\d k\dx=\int_0^1\lm^\be(1-\lm)^p\,\d\lm\int_{K_{\rho}}u^{p+\be+1}\,\dx.
\]
One could verify that there exists an absolute constant $C>0$, such that
\[
\int_0^1\lm^\be(1-\lm)^p\,\d\lm\ge\frac{C}{(\be+1)^{p+1}}.
\]
Similarly, the second term yields
\[
\iint_{Q_{\rho,\tau}}\int_0^{u}k^{\be}|Du|^p\,\dx\dt\,\d k=\frac{1}{\be+1}\iint_{Q_{\rho,\tau}}u^{\be+1}|Du|^p\,\dx\dt,
\]
while the integral on the right-hand side is estimated from above by
\begin{align*}
\frac{1}{\be+1}\iint_{Q_{R,T}}u^{p+\be+1}\,\dx\dt.
\end{align*}
Combining the above calculation gives us that for all $\be>-1$,
\begin{align*}
\frac1{(\be+1)^p}&\essup_{-\tau<t<0}\int_{K_{\rho}\times\{t\}}u^{p+\be+1}\,\dx+\iint_{Q_{\rho,\tau}}u^{\be+1}|Du|^p\,\dx\dt\\
&\le\bigg[\frac{\gm}{(R-\rho)^p}+\frac{\gm}{T-\tau}\bigg]\iint_{Q_{R,T}}u^{p+\be+1}\,\dx\dt.
\end{align*}
Written in terms of $w\df{=}u^{\frac{p+\be+1}{p}}$, the above estimate gives
\begin{align*}
\frac1{(\be+1)^p}&\essup_{-\tau<t<0}\int_{K_{\rho}\times\{t\}}w^{p}\,\dx
+\left(\frac{p}{p+\be+1}\right)^{p}\iint_{Q_{\rho,\tau}}|Dw|^p\,\dx\dt\\
&\le\bigg[\frac{\gm}{(R-\rho)^p}+\frac{\gm}{T-\tau}\bigg]\iint_{Q_{R,T}}w^{p}\,\dx\dt.
\end{align*}
This is the starting point of Moser's iteration scheme.
In order to use this energy estimate, we introduce for $\rho,\,\tau>0$, $\sig\in(0,1)$ and $n=0,1,\cdots$,
\begin{align*}
\left\{
\begin{array}{c}
\dsty\rho_n=\sig\rho+\frac{(1-\sig)\rho}{2^n},\quad \tau_n=\sig\tau+\frac{(1-\sig)\tau}{2^n},\\[5pt]
\dsty\tilde{\rho}_n=\frac{\rho_n+\rho_{n+1}}{2},\quad \tilde{\tau}_n=\frac{\tau_n+\tau_{n+1}}2,\\[5pt]
\dsty K_{n}=K_{\rho_n},\quad\tilde{K}_{n}=K_{\tilde{\rho}_n},
\quad Q_{n}=K_n\times (-\tau_n,0],\quad\tilde{Q}_n=\tilde{K}_n\times(-\tilde{\tau}_n,0],\\[5pt]
\dsty p_n=p+\be_{n}+1,\quad p_{n+1}=p_n\kappa,\quad\kappa=\frac{N+p}N,\quad\text{ i.e. }\,p_n=p_o\kappa^n.
\end{array}
\right.
\end{align*}
Set $\z$ to be a standard cutoff function that
vanishes on $\pl_{p}\tilde{Q}_n$ and equals identity in $Q_{n+1}$,
such that $|D\z|\le 2^n/\rho$. 
We apply the Sobolev imbedding (cf. \cite[Chapter I, Proposition~3.1]{DB-DPE}), together with the energy estimate
and the choice $p_o=p\kappa$ such that $\be_o>-1$,
to obtain
\begin{align*}
&\iint_{Q_{n+1}}u^{p_{n+1}}\,\dx\dt\le\iint_{\tilde{Q}_n}(w\z)^{p\frac{N+p}N}\,\dx\dt\\
&\le C\iint_{\tilde{Q}_n}|D(w\z)|^{p}\,\dx\dt\left(\essup_{-\tilde{\tau}_n<t<0}
\int_{\tilde{K}_n\times\{t\}}(w\z)^{p}\,\dx\right)^{\frac{p}N}\\
&\le\frac{C}{(1-\sig)^{p\kappa}}\left(\frac{p+\be_n+1}p\right)^{p}
(\be_n+1)^{\frac{p^2}N}\left(\frac{2^{p n}}{\rho^p}+\frac{2^n}{\tau}\right)^{\kappa}
\left(\iint_{Q_n}u^{p_n}\,\dx\dt\right)^{\kappa}\\
&\le \frac{C p_n^{p\kappa}2^{np\kappa}}{(1-\sig)^{p\kappa}}\left(\frac{1}{\rho^p}+\frac{1}{\tau}\right)^{\kappa}
\left(\iint_{Q_n}u^{p_n}\,\dx\dt\right)^{\kappa}\\
&\le \frac{Cb^{2np\kappa}}{(1-\sig)^{p\kappa}}\left(\frac{1}{\rho^p}+\frac{1}{\tau}\right)^{\kappa}\left(\iint_{Q_n}u^{p_n}\,\dx\dt\right)^{\kappa},
\end{align*}
for some $b,\, C>1$ depending only on the data.
To simply the above iteration, we set
\[
Y_n=\left(\frac1{|Q_n|}\iint_{Q_n}u^{p_n}\,\dx\dt\right)^{\frac1{p_n}},
\]
take the power $p_{n+1}^{-1}$ on both sides,
and rewrite it as
\[
Y_{n+1}\le B^{\frac{n}{\kappa^n}} 
 Y_n,
\]
where
\[
B=\frac{C}{(1-\sig)^{p\kappa}}
\left[\left(\frac{\tau}{\rho^p}\right)^{\frac{p}{N+p}}+\left(\frac{\rho^p}{\tau}\right)^{\frac{N}{N+p}}\right]^{\kappa}.
\]
Iterating this inequality yields 
\[
Y_n\le B^{\frac{n}{\kappa^n}+\frac{n-1}{\kappa^{n-1}}+\cdots+\frac1\kappa}Y_o\le B^{\frac{\kappa}{(\kappa-1)^2}}Y_o.
\]
Sending $n\to\infty$ gives
\[
\essup_{Q_{\sig\rho,\sig\tau}}u\le
\frac{C}{(1-\sig)^{\frac{p\kappa^2}{(\kappa-1)^2}}}\left[\left(\frac{\rho^p}{\tau}\right)^{\frac{N}{N+p}}+\left(\frac{\tau}{\rho^p}\right)^{\frac{p}{N+p}}\right]^{\frac{\kappa^2}{(\kappa-1)^2}}\left(\frac1{|Q_o|}\iint_{Q_o}u^{p_o}_+\,\dx\dt\right)^{\frac1{p_o}}.
\]
Define
\[
M=\essup_{Q_{\rho,\tau}}u_+,\qquad M_{\sig}=\essup_{Q_{\sig\rho,\sig\tau}}u_+.
\]
Then the above estimate yields
\[
M_\sig\le
\frac{CM^{1-\frac1{p\kappa}}}{(1-\sig)^{\frac{N+p}{p}}}\left[\left(\frac{\rho^p}{\tau}\right)^{\frac{N}{N+p}}+\left(\frac{\tau}{\rho^p}\right)^{\frac{p}{N+p}}\right]^{\frac{\kappa^2}{(\kappa-1)^2}}\left(\frac1{|Q_o|}\iint_{Q_o}u_+\,\dx\dt\right)^{\frac1{p\kappa}}.
\]
An interpolation argument would give 
\[
\essup_{Q_{\frac{\rho}2,\frac{\tau}2}}u_+\le C\left[\left(\frac{\rho^p}{\tau}\right)^{\frac{N}{N+p}}+\left(\frac{\tau}{\rho^p}\right)^{\frac{p}{N+p}}\right]^{\frac{p\kappa^3}{(\kappa-1)^2}}\dashint_{-\tau}^0\dashint_{K_\rho}u_+\,\dx\dt.
\]
Fixing $\sig_1,\,\sig_2\in(0,1)$, it is not hard to see that there exists $(y,s)\in Q_{\sig_1 R,\sig_2 T}$, such that
\[
\essup_{Q_{\sig_1 R,\sig_2 T}}u_+\le\essup_{Q_*}u_+,
\]
where we have set
\[
Q_*\df{=}(y,s)+Q_{\frac{(1-\sig_1)R}2,\frac{(1-\sig_2)T}2}.
\]
Applying the above estimate to $Q_*$ to obtain
\[
\essup_{Q_{\sig_1 R,\sig_2 T}}u_+\le\essup_{Q_*}u_+\le 
\frac{C\mathcal{A}}{(1-\sig_1)^NR^N(1-\sig_2)T}\int_{-T}^0\int_{K_R}u_+\,\dx\dt.
\]
where
\[
\mathcal{A}\df{=}\left\{\left[\frac{(1-\sig_1)^{p}R^p}{(1-\sig_2)T}\right]^{\frac{N}{(p+N)}}
+\left[\frac{(1-\sig_2)T}{(1-\sig_1)^pR^p}\right]^{\frac{p}{N+p}}\right\}^{\frac{p\kappa^3}{(\kappa-1)^2}}.
\]
Setting $\rho=\sig_1R$ and $\tau=\sig_2T$, the desired conclusion follows.
\subsection{Critical Mass Lemmas}
Assume $a\in(0,1)$ and $M>0$ are parameters.
The following lemma has been derived in \cite{GV}.
It can be viewed as a direct consequence of the local boundedness estimate
in Theorem~\ref{Thm:Bd}.
\begin{lemma}\label{Lm:critical-mass}
Let $u\in\mathfrak{A}^{\pm}_p(E;\gm)$. Suppose $(y,s)+Q_\rho\subset E$ and $\mu^{\pm}$ satisfy
\[
\mu^+\ge\essup_{(y,s)+Q_{\rho}}u,\qquad\mu^-\le\essinf_{(y,s)+Q_\rho}u.
\]
There exists $\nu>0$ depending only on the data and $a$, such that
if
\[
|[\pm(\mu^{\pm}-u)<M]\cap [(y,s)+Q_{\rho}]|\le\nu|Q_\rho|,
\]
then
\[
\pm(\mu^{\pm}-u)>aM\quad\text{ in }(y,s)+Q_{\frac{\rho}2}.
\]
\end{lemma}
\begin{proof}
Assume $(y,s)=(0,0)$. We only treat the class $\mathfrak{A}^+_p(E;\gm)$.
An application of Theorem~\ref{Thm:Bd} in $Q_{\frac{\rho}2}\Subset Q_\rho$, with $k=\mu^+-M$ yields that,
\begin{align*}
\essup_{Q_{\frac{\rho}2}}(u-k)_+\le \frac{C}{|Q_{\rho}|}\iint_{Q_\rho}(u-k)_+\,\dx\dt
\le CM\frac{|[u>k]\cap Q_\rho|}{|Q_\rho|}.
\end{align*}
Now we choose $\nu=\frac{1-a}{C}$, such that
when
\[
\frac{|[u>k]\cap Q_\rho|}{|Q_\rho|}<\nu,
\]
we have
\[
CM\frac{|[u>k]\cap Q_\rho|}{|Q_\rho|}<(1-a)M.
\]
As a result, we arrive at the desired conclusion
\[
\essup_{Q_{\frac{\rho}2}}u\le k+(1-a)M=\mu^+-aM.
\]
\end{proof}
\section{Additional Properties of Functions in $\mathfrak{A}_p^{\pm}$}\label{S:3}
It is known that the convex, non-decreasing function of a sub-harmonic function
yields another sub-harmonic function, whereas the concave, non-increasing function
of a super-harmonic function gives another super-harmonic function.
Similar conclusions hold for the heat operator, and even for more general linear parabolic operators
with bounded and measurable coefficients.
What we are concerned with next is to show analogous properties for 
members of $\mathfrak{A}^{\pm}_p(\gm,E)$.
\begin{lemma}\label{Lm:4:1}
Let $\vp: \rr\to\rr$ be convex and non-decreasing and 
let $u\in\mathfrak{A}_p^{+}(E,\gm)$. 
Then $\vp(u)$ belongs to the generalized class $\mathfrak{A}^+_p$.
\end{lemma}
\begin{proof} 
For any such $\vp$ and $h\le k$, observe the following elementary identity
\begin{equation}\label{Eq:4:1}
\big(\vp(u)-\vp(h)\big)_+ -\vp^{\prime}(h)(u-h)_+=\int_\rr (u-k)_+\chi_{[k>h]}\vp^{\prime\prime}(k)\,\d k,
\end{equation}
where $\chi$ is the characteristic function of the indicated set.
Moreover, by the convexity and monotonicity of $\vp$,
\begin{equation}\label{Eq:4:2}
\big(\vp(u)-\vp(h)\big)_+\ge\vp^{\prime}(h)(u-h)_+\ge0.
\end{equation}
From \eqref{Eq:4:1}, for a.e. $t\in(-\tau,0)$
\begin{align*}
\|\big(\vp&(u(\cdot, t))-\vp(h)\big)_+\|_{p,K_{\rho}}\\
&\le \|\vp'(h)(u(\cdot, t)-h)_+\|_{p,K_{\rho}}
+\left\|\int_\rr (u(\cdot, t)-k)_+\chi_{[k>h]}\vp^{\prime\prime}(k)\,\d k\right\|_{p,K_{\rho}}\\
&=I_1+I_2.
\end{align*}
For $I_1$, we estimate by using \eqref{Eq:4:1} and \eqref{Eq:4:2}:
\begin{align*}
I_1^p&\le \left[\frac{\gm}{(R-\rho)^p}+\frac{\gm}{(T-\tau)}\right]\iint_{Q_{R,T}}[\vp^{\prime}(h)]^p(u-h)_{+}^p\,\dx\dt\\
&\le \left[\frac{\gm}{(R-\rho)^p}+\frac{\gm}{(T-\tau)}\right]\iint_{Q_{R,T}}\big(\vp(u)-\vp(h)\big)_+^p\,\dx\dt.
\end{align*}
For $I_2$, we estimate by using \eqref{Eq:4:1}, \eqref{Eq:4:2} and Theorem~\ref{Thm:Bd}:
\begin{align*}
I_2&\le\int_\rr \|(u-k)_+\|_{p,K_{\rho}}\chi_{[k>h]}\vp^{\prime\prime}(k)\,\d k\\
&\le C\rho^{\frac{N}p}\int_\rr \|(u-k)_+\|_{\infty,K_{\rho}}\chi_{[k>h]}\vp^{\prime\prime}(k)\,\d k\\
&\le\frac{\mathcal{A}\rho^{\frac{N}p}}{(R-\rho)^N(T-\tau)}\int_\rr\iint_{Q_{R,T}}(u-k)_{+}\chi_{[k>h]}\vp^{\prime\prime}(k)\,\dx\dt\,\d k\\
&=\frac{\mathcal{A}\rho^{\frac{N}p}}{(R-\rho)^N(T-\tau)}\iint_{Q_{R,T}}[\big(\vp(u)-\vp(h)\big)_+-\vp'(h)(u-h)_+]\,\dx\dt\\
&\le\frac{\mathcal{A}\rho^{\frac{N}p}}{(R-\rho)^N(T-\tau)}\iint_{Q_{R,T}}\big(\vp(u)-\vp(h)\big)_+\,\dx\dt\\
&\le\frac{\mathcal{A}\rho^{\frac{N}p}(R^{N}T)^{1-\frac1p}}{(R-\rho)^N(T-\tau)}\|\big(\vp(u)-\vp(h)\big)_+\|_{p,Q_{R,T}}.
\end{align*}
Recalling $\mathcal{A}(R,T,\rho,\tau)$ represents a generic dimensionless quantity, 
we combine the above estimates to arrive at
\begin{align*}
\essup_{-\tau<t<0}&\int_{K_{\rho}}\big(\vp(u(\cdot,t))-\vp(k)\big)_{+}^p\,\dx\\
&\le\left[\frac{\gm}{(R-\rho)^p}+\frac{\gm+\mathcal{A}}{(T-\tau)}\right]
\iint_{Q_{R,T}}\big(\vp(u)-\vp(k)\big)_{+}^p\,\dx\dt.
\end{align*}

We now handle the part with the space gradient.
From  \eqref{Eq:4:1}, taking the gradient of both sides, then taking the $L^p$-norm over $Q_{\rho,\tau}$ and
applying the continuous version of Minkowski's inequality,
we obtain
\begin{align*}
\|D\big(\vp(u)-\vp(h)\big)_+\|_{p,Q_{\rho,\tau}}
&\le\|\vp'(h)D(u-h)_+\|_{p,Q_{\rho,\tau}}\\
&\quad+\left\|\int_\rr D(u-k)_+\chi_{[k>h]}\vp^{\prime\prime}(k)\,\d k\right\|_{p,Q_{\rho,\tau}}\\
&\le\|\vp^{\prime}(h)D(u-h)_+\|_{p,Q_{\rho,\tau}}\\
&\quad+\int_\rr \|D(u-k)_+\|_{p,Q_{\rho,\tau}}\chi_{[k>h]}\vp^{\prime\prime}(k)\,\d k\\
&=I_3+I_4.
\end{align*}
One estimates $I_3$ using \eqref{Eq:1:2} and \eqref{Eq:4:2}:
\begin{align*}
I_3^p&\le \left[\frac{\gm}{(R-\rho)^p}+\frac{\gm}{(T-\tau)}\right]\iint_{Q_{R,T}}(u-h)^p_+[\vp^\prime(h)]^p\,\dx\dt\\
&\le\left[\frac{\gm}{(R-\rho)^p}+\frac{\gm}{(T-\tau)}\right]\iint_{Q_{R,T}}\big(\vp(u)-\vp(h)\big)^p_+\,\dx\dt
\end{align*}
One estimates $I_4$ by \eqref{Eq:1:2}, \eqref{Eq:4:1}, \eqref{Eq:4:2} and Theorem~\ref{Thm:Bd}:
\begin{align*}
&\int_\rr \|D(u-k)_+\|_{p,Q_{\rho,\tau}}\vp^{\prime\prime}(k)\,\d k\\
&\le\left[\frac{\gm}{(R-\rho)^p}+\frac{\gm}{(T-\tau)}\right]^{\frac1p}
\int_\rr\|(u-k)_+\|_{p, Q_{\frac{R+\rho}2,\frac{T+\tau}2}}\chi_{[k>h]}\vp^{\prime\prime}(k)\,\d k\\
&\le\left[\frac{\gm}{(R-\rho)^p}+\frac{\gm}{(T-\tau)}\right]^{\frac1p}R^{\frac{N}p}T^{\frac1p}
\int_\rr \|(u-k)_+\|_{\infty, Q_{\frac{R+\rho}2,\frac{T+\tau}2}}\chi_{[k>h]}\vp^{\prime\prime}(k)\,\d k\\
&\le\left[\frac{\gm}{(R-\rho)^2}+\frac{\gm}{(T-\tau)}\right]^{\frac1p}\frac{\mathcal{A}R^{\frac{N}p}T^{\frac1p}}{(R-\rho)^N(T-\tau)}
\int_\rr\iint_{Q_{R,T}}(u-k)_+\chi_{[k>h]}\vp^{\prime\prime}(k)\,\d k\\
&=\left[\frac{\gm}{(R-\rho)^p}+\frac{\gm}{(T-\tau)}\right]^{\frac1p}\frac{\mathcal{A}R^{\frac{N}p}T^{\frac1p}}{(R-\rho)^N(T-\tau)}\\
&\quad\times
\iint_{Q_{R,T}}\big[\big(\vp(u)-\vp(h)\big)_+ -\vp^{\prime}(s)(u-h)_+\big]\,\dx\dt\\
&\le\left[\frac{\gm}{(R-\rho)^p}+\frac{\gm}{(T-\tau)}\right]^{\frac1p}\frac{\mathcal{A}R^{\frac{N}p}T^{\frac1p}}{(R-\rho)^N(T-\tau)}
\iint_{Q_{R,T}}\big(\vp(u)-\vp(h)\big)_+\,\dx\dt\\
&\le\left[\frac{\gm}{(R-\rho)^p}+\frac{\gm}{(T-\tau)}\right]^{\frac1p}\frac{\mathcal{A}R^{N}T}{(R-\rho)^N(T-\tau)}
\|\big(\vp(u)-\vp(h)\big)_+\|_{p,Q_{R,T}}.
\end{align*}
Observe that the fractional with $\mathcal{A}$
is again a dimensionless quantity.
Hence we have
\begin{align*}
\iint_{Q_{\rho,\tau}}&\left|D\big(\vp(u)-\vp(h)\big)_+\right|^p\,\dx\dt\\
&\quad\le\left[\frac{\mathcal{A}}{(R-\rho)^p}+\frac{\mathcal{A}}{(T-\tau)}\right]
\iint_{Q_{R,T}}\big(\vp(u)-\vp(h)\big)^p_+\,\dx\dt.
\end{align*}
Combining the above estimates gives the desired conclusion.
\end{proof}
\begin{lemma}\label{Lm:4:2}
Let $\vp: (a,\infty)\to\rr$, for some $a<\infty$ be convex and non-increasing,
such that
\begin{equation}\label{Eq:4:3}
\lim_{k\to\infty}\vp(k)=\lim_{k\to\infty}k\vp^{\prime}(k)=0.
\end{equation}
Suppose $u\in\mathfrak{A}_p^{-}(E,\gm)$, with range in $(a,\infty)$.
Then $\vp(u)$ belongs to the generalized class $\mathfrak{A}^+_p$.
\end{lemma}
\begin{proof}
Under the conditions of $\vp$, one easily verifies
\begin{equation}\label{Eq:4:4}
\vp(u)=\int_\rr(u-k)_-\vp^{\prime\prime}(k)\,\d k.
\end{equation}
Since $u\in\mathfrak{A}_p^-(E;\gm)$, it is bounded from below by Theorem~\ref{Thm:Bd}.
Hence the above equation is well defined for such $u$
and we may assume with no loss of generality that $u\ge0$.

First, we take $L^p$-norm of both sides over $K_\rho$ to obtain
for all $-\tau<t<0$
\[
\|\vp(u(\cdot,t))\|_{p, K_\rho}=\left\|\int_\rr(u(\cdot,t)-k)_-\vp^{\prime\prime}(k)\,\d k\right\|_{p,K_\rho}.
\]
The right-hand side is estimated by Minkowski's inequality and
Theorem~\ref{Thm:Bd}:
\begin{align*}
\int_\rr &\|(u(\cdot,t)-k)_-\|_{p,K_{\rho}}\vp^{\prime\prime}(k)\,\d k\\
&\le C\rho^{\frac{N}p}\int_\rr \|(u(\cdot, t)-k)_-\|_{\infty,K_{\rho}}\vp^{\prime\prime}(k)\,\d k\\
&\le\frac{\mathcal{A}\rho^{\frac{N}p}}{(R-\rho)^N(T-\tau)}\int_\rr\iint_{Q_{R,T}}(u-k)_{-}\vp^{\prime\prime}(k)\,\dx\dt\,\d k\\
&=\frac{\mathcal{A}\rho^{\frac{N}p}}{(R-\rho)^N(T-\tau)}\iint_{Q_{R,T}}\vp(u)\,\dx\dt\\
&\le\frac{\mathcal{A}\rho^{\frac{N}p}(R^{N}T)^{1-\frac1p}}{(R-\rho)^N(T-\tau)}\|\vp(u)\|_{p,Q_{R,T}}.
\end{align*}
As a result,
\[
\essup_{-\tau<t<0}\int_{K_{\rho}}|\vp(\cdot,t)|^p\,dx\le\frac{\mathcal{A}}{T-\tau}\iint_{Q_{R,T}}|\vp(u)|^p\,\dx\dt.
\]

Next, we take the spatial gradient of both sides of \eqref{Eq:4:4}, then take the power $p$,
and integrate over $Q_{\rho,\tau}$ to obtain
\[
\iint_{Q_{\rho,\tau}}|D\vp(u)|^p\,\dx\dt=
\iint_{Q_{\rho,\tau}}\left|\int_\rr D(u-k)_-\vp^{\prime\prime}(k)\,dk\right|^p\,\dx\dt.
\]
The right-hand side is estimated by
\begin{align*}
\|D&\vp(u)\|_{p,Q_{\rho,\tau}}\le\int_{\rr}\|D(u-k)_-\|_{p,Q_{\rho,\tau}}\vp^{\prime\prime}(k)\,\d k\\
&\le\left[\frac{\gm}{(R-\rho)^p}+\frac{\gm}{(T-\tau)}\right]^{\frac1p}
\int_\rr\|(u-k)_-\|_{p, Q_{\frac{R+\rho}2,\frac{T+\tau}2}}\vp''(k)\,\d k\\
&\le\left[\frac{\gm}{(R-\rho)^p}+\frac{\gm}{(T-\tau)}\right]^{\frac1p}R^{\frac{N}p}T^{\frac1p}
\int_\rr \|(u-k)_-\|_{\infty, Q_{\frac{R+\rho}2,\frac{T+\tau}2}}\vp''(k)\,\d k\\
&\le\left[\frac{\gm}{(R-\rho)^p}+\frac{\gm}{(T-\tau)}\right]^{\frac1p}\frac{\mathcal{A}R^{\frac{N}p}T^{\frac1p}}{(R-\rho)^N(T-\tau)}
\int_\rr\iint_{Q_{R,T}}(u-k)_-\vp''(k)\,\d k\\
&=\left[\frac{\gm}{(R-\rho)^p}+\frac{\gm}{(T-\tau)}\right]^{\frac1p}\frac{\mathcal{A}R^{\frac{N}p}T^{\frac1p}}{(R-\rho)^N(T-\tau)}
\iint_{Q_{R,T}}\vp(u)\,\dx\dt\\
&\le\left[\frac{\gm}{(R-\rho)^p}+\frac{\gm}{(T-\tau)}\right]^{\frac1p}\frac{\mathcal{A}R^{N}T}{(R-\rho)^N(T-\tau)}
\|\vp(u)\|_{p,Q_{R,T}}\\
&\le\left[\frac{\gm\mathcal{A}}{(R-\rho)^p}+\frac{\gm\mathcal{A}}{(T-\tau)}\right]^{\frac1p}
\|\vp(u)\|_{p,Q_{R,T}}.
\end{align*}
If $\vp$ is convex, non-increasing and satisfying \eqref{Eq:4:3},
then $(\vp-l)_+$ verifies the same properties. Hence the desired
conclusion is reached by replacing $\vp$ with $(\vp-l)_+$.
\end{proof}
\begin{lemma}\label{Lm:log}
Let $u\in\mathfrak{A}_p^{-}(E,\gm)$ be  non-negative and bounded above by a positive constant $M$.
Then 
\begin{equation*}
\iint_{(y,s)+Q_{\rho,\tau}}|D\ln u|^p\,\dx\dt\le\left[\frac{\gm p}{(R-\rho)^p}
+\frac{\gm p}{T-\tau}\right]\iint_{(y,s)+Q_{R,T}}\ln\frac{M}{u}\,\dx\dt
\end{equation*}
for any pair of cubes $(y,s)+Q_{\rho,\tau}\subset(y,s)+Q_{R,T}\subset E$.
\end{lemma}
\begin{proof}
Assume $(y,s)=(0,0)$. According to \eqref{Eq:1:1}, for all $0<k<M$,
\begin{align*}
\iint_{Q_{\rho,\tau}}|D(u-k)_{-}|^p\,\dx\dt
\le\bigg[\frac{\gm}{(R-\rho)^p}+\frac{\gm}{(T-\tau)}\bigg]\iint_{Q_{R,T}}(u-k)_{-}^p\,\dx\dt.
\end{align*}
To proceed, we multiply both sides by $k^{-p-1}$ and integrate from $0$ to $M$.
The left-hand side becomes
\begin{align*}
\int_0^M&\frac{\d k}{k^{p+1}}\iint_{Q_{\rho,\tau}}|D(u-k)_{-}|^p\,\dx\dt\\
&=\iint_{Q_{\rho,\tau}}\int_0^M|D(u-k)_{-}|^p\frac{\d k}{k^{p+1}}\,\dx\dt\\
&=\iint_{Q_{\rho,\tau}}|Du|^p\int_u^M\frac{\d k}{k^{p+1}}\,\dx\dt\\
&=\iint_{Q_{\rho,\tau}}\left(-\frac1p\frac{|Du|^p}{M^p}+\frac1p\frac{|Du|^p}{u^p}\right)\,\dx\dt\\
&=\iint_{Q_{\rho,\tau}}|D\ln u|^p\,\dx\dt-\frac1{p M^p}\iint_{Q_{\rho,\tau}}|Du|^p\,\dx\dt.
\end{align*}
The integral on the right-hand side is estimated by
\begin{align*}
\int_0^M&\frac{\d k}{k^{p+1}}\iint_{Q_{R,T}}(u-k)_-^p\,\dx\dt\\
&=\iint_{Q_{R,T}}\int_0^M\frac{(k-u)_+^p}{t^{p+1}}\,\d k\,\dx\dt\\
&=\iint_{Q_{R,T}}\left[-\frac1p\frac{(k-u)_+^p}{k^p}\bigg|_u^M+\int_u^M\frac{(t-u)^{p-1}}{t^{p-1}}\frac{dt}{t}\right]\,\dx\dt\\
&\le-\frac1{pM^p}\iint_{Q_{R,T}}(M-u)^p_+\,\dx\dt+\iint_{Q_{R,T}}\ln\frac{M}{u}\,\dx\dt.
\end{align*}
Hence combining the above two estimates we arrive at
\begin{align*}
&\iint_{Q_{\rho,\tau}}|D\ln u|^p\,\dx\dt\\
&\quad\le \frac1{M^p}\left\{
\iint_{Q_{R,T}}|Du|^p\,\dx dt
-\bigg[\frac{\gm}{(R-\rho)^p}+\frac{\gm}{(T-\tau)}\bigg]\iint_{Q_{R,T}}(u-M)_{-}^p\,\dx\dt\right\}\\
&\qquad+\bigg[\frac{\gm p}{(R-\rho)^p}+\frac{\gm p}{(T-\tau)}\bigg]\iint_{Q_{R,T}}\ln\frac{M}{u}\,\dx\dt.
\end{align*}
Since $u\in\mathfrak{A}_p^{-}(E,\gm)$, the term in the curly bracket is non-positive and can be discarded.
\end{proof}
\begin{remark}\upshape
The appearance of a logarithmic integral on the right-hand side is natural.
Suppose $0<u\le M$ is a super-solution to the heat equation. If we formally multiply the equation
by $-u^{-1}\z^2$ where $\z$ is a standard cutoff function in $Q_{\rho}$ 
vanishing on $\pl Q_{\rho}$. Then
an integration over $Q_{\rho}$ followed by
a standard calculation yields
\[
\iint_{Q_\rho}\z^2\pl_t \ln\frac{M}{u}\,\dx\dt+\iint_{Q_\rho}\left|D\ln u\right|^2\z^2\,\dx\dt\le2\iint_{Q_\rho}\z D\ln u D\z\,\dx\dt.
\]
A further integration by parts in time and an application of Young's inequality would give us
\[
\iint_{Q_\rho}\left|D\ln u\right|^2\z^2\,\dx\dt\le C\iint_{Q_{\rho}}|D\z|^2\,\dx\dt+C\iint_{Q_{\rho}}\z|\z_t| \ln\frac{M}{u}\,\dx\dt.
\]
\end{remark}
\section{Time Propagation of Positivity in Measure}\label{S:4}
In this section, we examine the role of \eqref{Eq:1:2}.
First of all, we present a standard lemma which says \eqref{Eq:1:2} alone
is sufficient to propagate positivity of $u$ in measure for a short period of time (cf. cite{LSU}).
\begin{proposition}\label{Prop:5:1}
Suppose $u$ is non-negative and satisfies $\eqref{Eq:1:2}_-$.
Assume for $M>0$ and $\al\in(0,1)$, we have $(s,s+\rho^p]\times K_\rho(y)\subset E$ and
\[
|[u(\cdot, s)>M]\cap K_{\rho}(y)|\ge\al |K_\rho|.
\]
Then there exist $\dl,\,\eps\in(0,1)$ depending only on the data and $\al$, such that
\[
|[u(\cdot, t)>\eps M]\cap K_{\rho}(y)|\ge\frac{\al}2|K_\rho|
\]
for all times
\[
s<t<s+\dl\rho^p.
\]
\end{proposition}
\begin{proof}
Assume $(y,s)=(0,0)$. We may apply $\eqref{Eq:1:2}_-$ with $k=M$
in the cylinders 
$$K_{(1-\sig)\rho}\times(0,\dl\rho^p]\subset K_{\rho}\times(0,\dl\rho^p]\df{=}Q_o;$$
in such a case, we have for all $0<t<\dl\rho^p$,
\begin{align*}
\int_{K_{(1-\sig)\rho}} (u(\cdot,t)-M)^p_-\,\dx&\le\int_{K_\rho} (u(x,0)-M)^p_-\,\dx
+\frac{\gm}{(\sig\rho)^p}\iint_{Q_o}(u-M)^{p}_-\,\dx\dt\\
&\le\int_{K_\rho} (u(x,0)-M)^p_-\,\dx+\gm\frac{k^p}{(\sig\rho)^p}|[u<M]\cap Q_o|\\
&\le M^p\left[1-\al+\gm\frac{\dl}{\sig^p}\frac{|[u<M]\cap Q_o|}{|Q_o|}\right]|K_\rho|.
\end{align*}

Set $l=\eps M$. The left-hand side of the above estimate can be bounded from below by
\begin{align*}
\int_{K_{(1-\sig)\rho}\cap[u\le l]} (u(\cdot,t)-M)^p_-\,\dx\ge (1-\eps)^p M^{p}|A_{l,(1-\sig)\rho}(t)|
\end{align*}
where we have defined for some $\eps$ to be chosen
\[
A_{l,(1-\sig)\rho}(t)=[u(\cdot,t)\le\eps M]\cap K_{(1-\sig)\rho}.
\]
Notice that
\begin{align*}
|A_{l,\rho}(t)|&=|A_{l,(1-\sig)\rho}(t)\cup (A_{l,\rho}(t)-A_{l,(1-\sig)\rho}(t))|\\
&\le |A_{l,(1-\sig)\rho}(t)|+|K_\rho- K_{(1-\sig)\rho}|\\
&\le |A_{l,(1-\sig)\rho}(t)|+N\sig |K_\rho|.
\end{align*}
Collecting all the above estimates yields that
\begin{equation}\label{Eq:5:0}
|A_{l,\rho}(t)|\le \frac{1-\al}{(1-\eps)^p}|K_\rho|
+C\frac{\dl}{\sig^p}\frac{|[u<M]\cap Q_o|}{|Q_o|}|K_\rho|
+N\sig|K_\rho|
\end{equation}
Finally we may choose $\eps$, $\sig$ and $\dl$, such that
\[
\frac{1-\al}{(1-\eps)^p}\le1-\frac{3}4\al,\quad N\sig=\frac{\al}8,\quad C\frac{\dl}{\sig^p}\le\frac{\al}8.
\]
\end{proof}
\begin{remark}\label{Rmk:5:1}\upshape
One easily obtains the dependence of various constants on $\al$ from  the above proof.
Namely, $\eps\approx\al$, $\sig\approx\al$ and $\dl\approx\al^{p+1}$.
\end{remark}
One wonders if the positivity in measure can be propagated
further in time, i.e., $\dl$ can be made large by choosing
a proper $\eps$. It seems $\eqref{Eq:1:2}_-$ alone is insufficient.
In the theory of parabolic equations, a standard tool to achieve this is
a logarithmic estimate. See \cite[Chapter 2, Section 3]{DB}.
We do not know if such a logarithmic estimate holds for functions in parabolic De Giorgi
classes. However we show in the following that a membership in $u\in\mathfrak{B}^-_p(E,\gm)$
still ensures that the measure information of positivity propagates further in time.
\begin{proposition}\label{Prop:5:2}
Suppose $u\in\mathfrak{B}^-_p(E,\gm)$ is non-negative.
Assume for $A,\,M>0$ and $\al\in(0,1)$, we have $(s,s+A\rho^p]\times K_\rho(y)\subset E$ and
\[
|[u(\cdot, s)>M]\cap K_{\rho}(y)|\ge\al |K_\rho|.
\]
Then there exist $\eps>0$ depending on the data and $\al$, such that
\[
|[u(\cdot, t)>\eps M]\cap K_{\rho}(y)|\ge\frac{\al}2|K_\rho|
\]
for all
\[
s<t<s+A\rho^p.
\]
\end{proposition}
\subsection{Shrinking the Measure of the Set $[u\approx0]$}
We first prove the following shrinking lemma due to De Giorgi (cf. \cite{DG}).
\begin{lemma}\label{Lm:5:1}
Let $\al,\,\dl\in(0,1)$.
Suppose there holds
\[
\left|\left[u(\cdot, t)>M\right]\cap K_{\rho}\right|\ge\al |K_\rho|
\quad\text{ for all }t\in(s,s+\dl\rho^p].
\]
There exists $C>0$ depending only on the data, such that for any positive integer $j_*$,
we have
\[
\left|\left[u\le\frac{M}{2^{j_*}}\right]\cap Q\right|\le\frac{C}{\al \dl^{\frac1p} j_*^{\frac{p-1}p}}|Q|,\quad
\text{ where }Q=K_{\rho}\times\left(s,s+\dl\rho^p\right].
\]
\end{lemma}
\begin{proof}
We assume $(y,s)=(0,0)$ and set $k_j=2^{-j}M$ for $j=0,1,\cdots, j_*$. Apply $\eqref{Eq:1:1}_-$ for 
the pair of cylinders
\[
K_{\rho}\times(0,\dl\rho^p]\subset K_{2\rho}\times(-\dl\rho^p,\dl\rho^p],
\]
such that
\begin{equation}\label{Eq:5:1}
\iint_{Q}|D(u-k_j)_-|^p\,\dx\dt\le\frac{C}{\dl\rho^p}\left(\frac{M}{2^j}\right)^{p}|Q|.
\end{equation}
Next, we apply \cite[Chapter I, Lemma 2.2]{DB}
to $u(\cdot,t)$ for 
$t\in\left(0,\dl\rho^p\right]$
 over the cube $K_\rho$,
for levels $k_{j+1}<k_{j}$.
Taking into account the measure theoretical information
\[
\left|\left[u(\cdot, t)>M\right]\cap K_{\rho}\right|\ge\al |K_\rho|
\quad\text{ for all }t\in(0,\bar{\dl}\theta\rho^2],
\]
this gives
\begin{align*}
\frac{M}{2^{j+1}}&|[u(\cdot,t)<k_{j+1}]\cap K_\rho|\\
&\le\frac{C \rho^{N+1}}{|[u(\cdot,t)>k_j]\cap K_\rho|}\int_{[k_j<u(\cdot,t)<k_{j+1}]\cap K_{\rho}}|Du|\,\dx\\
&\le\frac{C\rho}{\al}\bigg(\int_{[k_j<u(\cdot,t)<k_{j+1}]\cap K_{\rho}}|Du|^p\,\dx\bigg)^{\frac1p}\\
&\quad\times|([u(\cdot,t)<k_j]-[u(\cdot,t)<k_{j+1}])\cap K_\rho|^{\frac{p-1}p}.
\end{align*}
Set
\[
A_j=[u<k_j]\cap Q
\]
and integrate the above estimate in $\dt$ over $(0,\dl\rho^p]$;
we obtain by using \eqref{Eq:5:1} 
\begin{align*}
\frac{M}{2^j}|A_{j+1}|&\le\frac{C\rho}{\al}\bigg(\iint_{Q}|D(u-k_j)_-|^p\,\dx\dt\bigg)^\frac1p(|A_j|-|A_{j+1}|)^\frac{p-1}p\\
&\le\frac{C}{\al\dl^{\frac1p}}\frac{M}{2^j}|Q|^{\frac1p}(|A_j|-|A_{j+1}|)^\frac{p-1}p.
\end{align*}
Now take the power $\frac{p}{p-1}$ on both sides of the above inequality to obtain
\[|A_{j+1}|^{\frac{p}{p-1}}\le\frac{C}{\al^{\frac{p}{p-1}}\dl^{\frac1{p-1}}}|Q|(|A_j|-|A_{j+1}|).\]
Add these inequalities from $0$ to $j_*-1$ to obtain
\[
j_* |A_{j_*}|^{\frac{p}{p-1}}\le\sum_{j=0}^{j_*-1}|A_{j+1}|^{\frac{p}{p-1}}\le\frac{C}{\al^{\frac{p}{p-1}}\dl^{\frac1{p-1}}}|Q|^2.
\]
From this we conclude
\[|A_{j_*}|\le\frac{C}{\al \dl^{\frac1p} j_*^{\frac{p-1}p}}|Q|.\]
\end{proof}
\subsection{Proof of Proposition~\ref{Prop:5:2}}
 We come back at \eqref{Eq:5:0} and choose
\begin{align*}
\sig=\dl^{\frac1{p+1}}\left(\frac{|[u<k]\cap Q_o|}{|Q_o|}\right)^{\frac1{p+1}},
\end{align*}
such that \eqref{Eq:5:0} becomes
\[
|A_{l,\rho}(t)|\le\bigg[\frac{1-\al}{(1-\eps)^p}
+C\dl^{\frac1{p+1}}\left(\frac{|[u<k]\cap Q_o|}{|Q_o|}\right)^{\frac1{p+1}}\bigg]|K_\rho|.
\]
We choose $\dl$ and $\eps$ such that
\begin{equation}\label{Eq:5:2}
C\dl^{\frac1{p+1}}=\frac{\al}{8},
\quad\frac{1-\al}{(1-\eps)^p}<\frac{1-\frac12\al}{(1-\eps)^p}\le1-\frac{1}{4}\al.
\end{equation}
As a result, we obtain
\[
|A_{l,\rho}(t)|\le\left(1-\frac{\al}{8}\right)|K_\rho|\quad\text{ for all }s\le t\le s_1\df{=}s+\dl\rho^p.
\]
Having $\eps$ and $\dl$
determined in \eqref{Eq:5:2}, we use $\eqref{Eq:1:2}_-$ again and
repeat the above argument with
 \[
 M_1=\eps M,\quad l_1=\frac{M_1}{2^{n_1+j_1}},\quad
 k_1=\frac{M_1}{2^{j_1}},
 \]
 where $j_1$ and $n_1$ are positive numbers to be determined.
 We may use the above measure theoretical information
 for $t\in[s,s_1]$, and apply Lemma~\ref{Lm:5:1}
 to obtain a refined estimate:
 \begin{align*}
|A_{l_1,\rho}(t)|
\le\bigg[\frac{1-\al}{(1- 2^{-n_1})^2}
+C\dl^{\frac1{p+1}}\left(\frac{1}{\al\dl^{\frac1p}j_1^{\frac{p-1}p}}\right)^{\frac1{p+1}}\bigg]|K_\rho|
\quad\text{ for all } t\in[0, s_1].
\end{align*}
We choose $j_1$ and $n_1$,
such that
\[
C\dl^{\frac1{p+1}}\left(\frac{1}{\al\dl^{\frac1p}j_1^{\frac{p-1}p}}\right)^{\frac1{p+1}}\le\frac{\dl\al}{4A},
\quad\frac{1-\al}{(1-2^{-n_1})^2}\le1-\al+\frac{\dl\al}{4A}.
\]
As a result, we obtain that 
\[
|A_{l_1,\rho}(t)|\le\left(1-\al+\frac{\dl\al}{2A}\right)|K_\rho|
\quad\text{ for all }t\in[s,s_1].
\]
Now we may proceed by induction. Suppose the construction has been made 
up to the $(i-1)$-th step:
the sequences $\{M_i\}$, $\{n_i\}$ and $\{j_i\}$ have been chosen up to the $(i-1)$-th step,
and we have the measure theoretical information
\[
|A_{l_{i-1},\rho}(t)|\le\left(1-\al+(i-1)\frac{\dl\al}{2A}\right)|K_\rho|
\quad\text{ for all }t\in[s_{i-1},s_i],
\]
where
\[
 l_{i-1}=\frac{ M_{i-1}}{2^{n_{i-1}+j_{i-1}}}\df{=}\widehat{M}_{i-1}.
\]
Setting 
\[
l^{\eps}_{i-1}=\eps\widehat{M}_{i-1},\quad
 s_{i+1}=s_i+\dl\rho^{p},\quad Q_i=K_{\rho}\times (s_i, s_{i+1}],
\]
and using the above measure theoretical information at $t=s_i$,
we can repeat the above argument to obtain for all $t\in[s_{i},s_{i+1}]$
\[
|A_{l^{\eps}_{i-1},\rho}(t)|\le\bigg[\frac{1-\al+(i-1)\frac1{2A}\dl\al}{(1-\eps)^2}
+C\dl^{\frac1{p+1}}\left(\frac{|[u<l_{i-1}]\cap Q_i|}{|Q_i|}\right)^{\frac1{p+1}}\bigg]|K_\rho|.
\]
Assuming $(i-1)\dl<A$, we may choose $\eps$ and $\dl$ as in \eqref{Eq:5:2}; this ensures
\[
|A_{l^{\eps}_{i-1},\rho}(t)|\le\left(1-\frac{\al}{8}\right)|K_\rho|\quad\text{ for all }t\in[s_i,s_{i+1}].
\]
Now we set
 \begin{align*}
 & M_i=\eps\widehat{M}_{i-1},\quad l_i=\frac{ M_i}{2^{n_i+j_i}},\quad
  k_i=\frac{M_i}{2^{j_i}},
 \end{align*}
  where $j_i$ and $n_i$ are to be determined.
Then we use the above measure theoretical information in Lemma~\ref{Lm:5:1}
 to obtain a refined estimate: for all $t\in[s_i,s_{i+1}]$
\begin{align*}
|A_{l_i,\rho}(t)|&\le\bigg[\frac{1-\al+(i-1)\frac{1}{2A}\dl\al}{(1-2^{-n_i})^2}
+C\dl^{\frac1{p+1}}\left(\frac{1}{\al\dl^{\frac1p}j_i^{\frac{p-1}p}}\right)^{\frac1{p+1}}\bigg]|K_\rho|.
\end{align*}
We choose $j_i$ and $n_i$,
such that
\begin{align*}
&C\dl^{\frac1{p+1}}\left(\frac{1}{\al\dl^{\frac1p}j_i^{\frac{p-1}p}}\right)^{\frac1{p+1}}\le\frac{\dl\al}{4A},\\
&\frac{1-\al+(i-1)\frac{1}{2}\dl\al}{(1-2^{-n_i})^2}\le1-\al
+(i-1)\frac{\dl\al}{2}+\frac{\dl\al}{4A}.
\end{align*}
As a result, we obtain that for all times $t\in[s_i,s_{i+1}]$
\[
|A_{l_i,\rho}(t)|\le\left(1-\al+i\frac{\dl\al}{2A}\right)|K_\rho|.
\]
The above argument terminates if $i\dl\ge A$ and we reach the desired conclusion
with the choice
\[\eps M=\frac{ M_i}{2^{n_i+j_i}}.\]
\section{H\"older Continuity for Functions in $\mathfrak{B}_p$}\label{S:5}
\begin{theorem}\label{Thm:Holder}
If $u\in \mathfrak{B}_p(E;\gm)$, there are constants $C>0$ and $0<\be<1$ depending only on the data, such that
for every pair of cylinders $(y,s)+Q_\rho\Subset (y,s)+Q_R\Subset E$, we have
\begin{equation*}
\essosc_{(y,s)+Q_\rho}u\le C\essosc_{(y,s)+Q_R}u\cdot\bigg(\frac{\rho}{R}\bigg)^{\be}
\end{equation*}
\end{theorem}
\noi For a function $u\in \mathfrak{B}_p(E;\gm)$ and $(y,s)+Q_{2\rho}\Subset E$, we set
\begin{equation*}
\mu^+=\essup_{(y,s)+Q_{2\rho}}u,\quad\mu^-=\essinf_{(y,s)+Q_{2\rho}}u,
\quad\om(2\rho)=\essosc_{(y,s)+Q_{2\rho}}u=\mu^+-\mu^-.
\end{equation*}
\subsection{Proof by Moser's Approach}
The purpose of this section is to prove Theorem~\ref{Thm:Holder}
using an intuitive idea of Moser. Thus the heavy machinery of De Giorgi, such as
Lemma~\ref{Lm:critical-mass} and Lemma~\ref{Lm:5:1}, is avoided.
A similar adaption has been made to parabolic equations in \cite{Le}, which
however cannot be directly generalized to parabolic De Giorgi classes. 

Without loss of generality, we may take $(y,s)=(0,0)$.
For ease of notation, we write $\om=\om(2\rho)$. Let $\eps$ be the number determined
in Proposition~\ref{Prop:5:1} with $\al=1/2$. We introduce two functions:
\begin{equation*}
w_1\df{=}\vp_1(u)=\ln \frac{\eps\om}{2(\mu^+-u)},\qquad w_2\df{=}\vp_2(u)=\ln \frac{\eps\om}{2(u-\mu^-)}.
\end{equation*}
We first apply Lemma~\ref{Lm:log} to $w_1$ and $w_2$.
Indeed, since $u\in\mathfrak{A}_p(E;\gm)$ we have both $\mu^+-u$
and $u-\mu^-$ members of $\mathfrak{A}_p^-(E;\gm)$. 
 Therefore, we may apply Lemma~\ref{Lm:log} to $\mu^+-u$ in $Q_R$ with $\rho<R<2\rho$,
to obtain
\[
\iint_{Q_{\rho}}\bigg|D\ln \frac{\eps\om}{2(\mu^+-u)}\bigg|^p\,\dx\dt
\le\frac{C}{(R-\rho)^p}\iint_{Q_{R}}\ln\frac{\om}{\mu^+-u}\,\dx\dt,
\]
that is, in terms of $w_1$,
\begin{equation}\label{Eq:6:1}
\iint_{Q_{\rho}}|Dw_1|^p\,\dx\dt\le\frac{C}{(R-\rho)^p}\iint_{Q_{R}}|w_1|\,\dx\dt+\frac{CR^{N+p}}{(R-\rho)^p}.
\end{equation}
Similar inequality holds for $w_2$.

Now we go with two alternatives: either 
\[
\left|\left[\mu^+-u(\cdot, -\dl\rho^p)\ge\frac{\om}2\right]\cap K_{\rho}\right|\ge\frac12|K_\rho|,
\]
or
\[
\left|\left[u(\cdot, -\dl\rho^p)- \mu^-\ge\frac{\om}2\right]\cap K_{\rho}\right|\ge\frac12|K_\rho|,
\]
where $\dl$ is the constant appearing in Proposition~\ref{Prop:5:1}
with $\al=1/2$.
Let us suppose for instance the first case holds.
According to Proposition~\ref{Prop:5:1}, we have
\[
\left|\left[\mu^+-u(\cdot, t)\ge\frac{\eps\om}2\right]\cap K_{\rho}\right|\ge\frac14|K_\rho|
\quad\text{ for all } -\dl\rho^p<t<0.
\]
In terms of $w_1$, this may be rephrased as 
\[
\left|\left[w_1(\cdot, t)\le 0\right]\cap K_\rho\right|\ge\frac14|K_\rho|\quad\text{ for all } -\dl\rho^p<t<0.
\]
We may employ the Poincar\'e type inequality (cf. \cite[Chapter 10, Proposition 5.2]{DB}) for each time slice to $w_1(\cdot, t)$, 
and then a time integration over $(-\dl\rho^p,0)$ on both sides,
 and the fact that
$w_1\ge-\ln(2/\eps)$ to obtain that
\begin{align*}
\int_{-\dl\rho^p}^0\int_{K_{\rho}}|w_1|\,\dx\dt&=\int_{-\dl\rho^p}^0\int_{K_\rho}w_{1+}\,\dx\dt
+\int_{-\dl\rho^p}^0\int_{K_\rho}w_{1-}\,\dx\dt\\
&\le C\rho\int_{-\dl\rho^p}^0\int_{K_\rho}|Dw_{1+}|\,\dx\dt+C\rho^{N+p}.
\end{align*}
The integral term on the right-hand side is estimated by H\"older's inequality, Young's inequality and \eqref{Eq:6:1} as
\begin{align*}
C\rho&\int_{-\dl\rho^p}^0\int_{K_\rho}|Dw_{1+}|\,\dx\dt\\
&\le C \rho^{1+N+p-\frac{N+p}p}\bigg(\iint_{Q_\rho}|Dw_{1+}|^p\,\dx\dt\bigg)^{\frac1p}\\
&\le C \rho^{1+N+p-\frac{N+p}p}\bigg(\frac{C}{(R-\rho)^p}\iint_{Q_{R}}|w_1|\,\dx\dt+\frac{CR^{N+p}}{(R-\rho)^p}\bigg)^{\frac1p}\\
&\le C \frac{\rho^{1+N+p-\frac{N+p}p}}{R-\rho}\bigg(\iint_{Q_R}|w_1|\,\dx\dt\bigg)^{\frac1p}+
C\frac{\rho^{1+N+p}}{R-\rho}\bigg(\frac{R}{\rho}\bigg)^{\frac{N+p}p}.\\
\end{align*}
Thus combining above estimates we obtain
\begin{equation*}
\begin{aligned}
\iint_{Q_{\dl\rho}}|w_1|\,\dx\dt&\le 
C \frac{\rho^{1+N+p-\frac{N+p}p}}{R-\rho}\bigg(\iint_{Q_R}|w_1|\,\dx\dt\bigg)^{\frac1p}\\
&\quad+C\frac{\rho^{1+N+p}}{R-\rho}\bigg(\frac{R}{\rho}\bigg)^{\frac{N+p}p}+C\rho^{N+p}.
\end{aligned}
\end{equation*}
An interpolation argument (cf. \cite[Theorem 1]{DT}) yields that
\begin{equation}\label{Eq:3:3}
\frac{1}{(\dl\rho)^{N+p}}\iint_{Q_{\dl\rho}}|w_1|\,\dx\dt\le C({\rm data}).
\end{equation}

An application of Lemma \ref{Lm:4:1} gives that $w_{1+}$ belongs to the generalized $\mathfrak{A}_p^+$.
As a result, Theorem \ref{Thm:Bd} holds for $w_{1+}$. The supreme estimate together
with \eqref{Eq:3:3} yields that
\[
\essup_{Q_{\frac{\dl\rho}2}}w_{1+}\le \frac{C}{(\dl\rho)^{N+p}}\iint_{Q_{\dl\rho}}|w_1|\,\dx\dt\le C({\rm data}),
\]
which implies
\[
\essup_{Q_{\frac{\dl\rho}2}}u\le \mu^+-\frac{\eps}{2e^{C}}\om.
\]
Therefore
\[
\essosc_{Q_{\frac{\dl\rho}2}}u\le\left(1-\frac{\eps}{2e^{C}}\right)\om.
\]
A standard iteration finishes the proof.
\subsection{A Revisit to De Giorgi's Approach}
The purpose of this section is to point out that the H\"older regularity
could be established with less assumptions. Namely, it suffices
to assume $u$ is a member of $\mathfrak{A}^+_p(E;\gm)\cap\mathfrak{B}^-_p(E;\gm)$
or $\mathfrak{A}^-_p(E;\gm)\cap\mathfrak{B}^+_p(E;\gm)$.
The argument is modeled on the one in \cite[Chapter III]{DB-DPE}.
\subsubsection{Expansion of Positivity}
Suppose $K_{4\rho}(y)\times(s,s+\rho^p]\subset E$.
We show in the following that the measure information on the positivity 
of a non-negative member $u$ of $\mathfrak{B}_p^-(E;\gm)$ at $t=s$ translates into pointwise
information forward in time and over a larger space cube.
\begin{proposition}\label{Prop:6:1}
Let $u\in\mathfrak{B}_p^-(E;\gm)$ be non-negative.
Suppose for some $M>0$ and $\al\in(0,1)$,
\[
|[u(\cdot, s)>M]\cap K_\rho(y)|\ge\al |K_\rho|.
\]
Then there exist $\eta,\,\dl\in(0,1)$ depending on the data and $\al$, such that
\[
u(\cdot, t)\ge\eta M\quad\text{ a.e. in }K_{2\rho}(y),
\]
for all 
\[
s+\frac12\dl(4\rho)^p<t<s+\dl(4\rho)^p.
\]
\end{proposition}
\begin{proof}
Assume $(y,s)=(0,0)$. We rephrase the starting information in a larger cube:
\[
|[u(\cdot, 0)>M]\cap K_{4\rho}(y)|\ge\al 4^{-N}|K_{4\rho}|.
\]
By Proposition~\ref{Prop:5:1},
there exist $\dl,\,\eps>0$ depending only on the data and $\al$, such that
\[
|[u(\cdot, t)>\eps M]\cap K_{4\rho}|\ge\frac{\al}2 4^{-N}|K_{4\rho}|
\]
for all
\[
0<t<\dl(4\rho)^p.
\]
Next, by Lemma~\ref{Lm:5:1}, there exists $C>0$ depending only on the data, such that for any positive integer $j_*$,
we have
\[
\left|\left[u\le\frac{\eps M}{2^{j_*}}\right]\cap Q\right|\le\frac{C}{\al \dl^{\frac1p} j_*^{\frac{p-1}p}}|Q|,\quad
\text{ where }Q=K_{4\rho}\times\left(0,\dl(4\rho)^p\right].
\]
Finally, let $\nu$ be the number claimed in Lemma~\ref{Lm:critical-mass}.
Choose $j_*$ so large that
\[
\frac{C}{\al \dl^{\frac1p} j_*^{\frac{p-1}p}}\le\nu.
\]
Thus by Lemma~\ref{Lm:critical-mass} with $\mu^-=0$, we conclude that
\[
u(\cdot, t)\ge \frac{\eps M}{2^{j_*+1}}\quad\text{ a.e. in }K_{2\rho}
\]
for all times
\[
\frac12\dl(4\rho)^p<t<\dl(4\rho)^p.
\]
The proof is finished by choosing $\eta=\eps 2^{-j_*-1}$.
\end{proof}
\begin{remark}\label{Rmk:5:1}\upshape
By repeated applications of Proposition~\ref{Prop:6:1},
for any $A>0$ there exist $\bar\eta\in(0,1)$ depending on the data, $\al$ and $A$,
such that
\[
u(\cdot, t)\ge\bar\eta M\quad\text{ a.e. in }K_{2\rho}(y),
\]
for all times
\[
s+\rho^p<t<s+A\rho^p.
\]
\end{remark}
\subsubsection{Another Proof of Theorem~\ref{Thm:Holder}}
Let $\nu>0$ be the number fixed in Lemma~\ref{Lm:critical-mass} with $a=1/2$,
and suppose 
\begin{equation*}
\left|\left[\mu^+-u\ge\frac{\om}2\right]\cap Q_{\rho}\right|\le\nu|Q_{\rho}|.
\end{equation*}
Then since $u\in\mathfrak{A}^+_p(E;\gm)$, Lemma~\ref{Lm:critical-mass} would give us that
\[
\mu^+-u\ge\frac{\om}4\quad\text{ a.e. in }Q_{\frac{\rho}2},
\]
which in turn gives the reduction of oscillation:
\[
\essosc_{Q_{\frac{\rho}2}}u\le\frac34\om.
\]
Now suppose to the contrary that
\begin{equation*}
\left|\left[\mu^+-u\ge\frac{\om}2\right]\cap Q_{\rho}\right|>\nu|Q_{\rho}|.
\end{equation*}
Then there exists some 
\[-\rho^p\le s\le-\frac{\nu}2\rho^p,\]
such that 
\begin{equation*}
\left|\left[\mu^+-u(\cdot, s)\ge\frac{\om}2\right]\cap K_{\rho}\right|>\frac{\nu}2|K_{\rho}|.
\end{equation*}
Indeed, if the above inequality does not hold for any $s$ in the given 
interval, then 
\begin{align*}
\left|\left[\mu^+-u\ge\frac{\om}2\right]\cap Q_{\rho}\right|&=
\int_{-\rho^p}^{-\frac12\nu\rho^p}\left|\left[\mu^+-u(\cdot, s)\ge\frac{\om}2\right]\cap K_{\rho}\right|\,ds\\
&\quad+\int^{0}_{-\frac12\nu\rho^p}\left|\left[\mu^+-u(\cdot, s)\ge\frac{\om}2\right]\cap K_{\rho}\right|\,ds\\
&\le\nu |Q_{\rho}|.
\end{align*}
Since $\mu^+-\frac12\om>\mu^-+\frac12\om$ always holds, this implies
\begin{equation*}
\left|\left[u(\cdot,s)-\mu^->\frac{\om}2\right]\cap K_{\rho}\right|\ge\frac{\nu}2|K_{\rho}|.
\end{equation*}
Then since $u\in\mathfrak{B}^-_p(E;\gm)$, Proposition~\ref{Prop:6:1} (see also Remark~\ref{Rmk:5:1}) gives
$\eta\in(0,1)$ depending only on the data, such that
\[
u(\cdot,s)-\mu^->\eta\om\quad\text{ a.e. in }Q_{\frac{\rho}2},
\]
which in turn gives
\[
\essosc_{Q_{\frac{\rho}2}}u\le(1-\eta)\om.
\]
\section{Harnack's Inequalities for Functions in $\mathfrak{B}_p$}\label{S:6}
Assume
\[
K_{4\rho}(y)\times[s-(4\rho)^p,s+(4\rho)^p]\Subset E.
\]
The following Harnack's inequality has been shown in \cite{GV}. See also \cite{W}.
\begin{theorem}\label{Thm:Harnack}
Let $u\in\mathfrak{B}_p(E;\gm)$ be non-negative.
There exist $\theta\in(0,1)$ and $C>1$ depending only on the data,
such that
\[
C^{-1}\sup_{K_{\rho}(y)}u(\cdot,s-\theta\rho^p)\le u(y,s)\le C\inf_{K_{\rho}(y)}u(\cdot,s+\theta\rho^p).
\]
\end{theorem}
The approach used in \cite{GV} is a direct one, thus by-passing a weak Harnack inequality.
On the other hand, a weak Harnack inequality is established in \cite{W} for $p=2$
using the Krylov-Safonov covering argument (cf. \cite{KS}). Here we give a transparent proof of a weak
Harnack inequality for the class $\mathfrak{B}^-_p(E,\gm)$, via a measure theoretical 
lemma in \cite{DBGV}, thus avoiding the heavy covering argument.
\subsection{Weak Harnack Inequality for Functions in $\mathfrak{B}^-_p(E,\gm)$}
Assume the cylinder
\[
K_{4\rho}\times(s,s+(4\rho)^p]\Subset E.
\]
\begin{theorem}\label{Thm:7:2}
Let $u\in\mathfrak{B}_p^-(E;\gm)$ be non-negative.
Then there exist $\dl_o,\, q\in(0,1)$ and $C>1$ depending only on the data, such that
\[
C\essinf_{K_{2\rho}(y)}u(\cdot,t)\ge \left(\dashint_{K_{\rho}(y)} u^q(\cdot,s)\,\dx\right)^{\frac1{q}},
\]
for all times
\[
s+\frac12\dl_o\rho^p<t<s+\dl_o\rho^p.
\]
\end{theorem}
A combination of Theorem~\ref{Thm:7:2} and Theorem~\ref{Thm:Bd}
would give another proof of Theorem~\ref{Thm:Harnack} (cf. \cite{DT, W}).
The key to proving Theorem~\ref{Thm:7:2} is to show an expansion of
positivity with a power-like dependence on the measure distribution
of the positivity. The main tool is a certain clustering lemma  from \cite{DBGV}.
\begin{proposition}
Let $u\in\mathfrak{B}_p^-(E;\gm)$ be non-negative.
Suppose for some $M>0$ and $\al\in(0,1)$,
\[
|[u(\cdot, s)>M]\cap K_\rho(y)|\ge\al |K_\rho|.
\]
Then there exist $\dl_o,\,\eta_o\in(0,1)$ and $d>1$ depending only on the data, such that
\[
u(\cdot, t)\ge\eta_o\al^d M\quad\text{ a.e. in }K_{2\rho}(y),
\]
for all times
\[
s+\frac12\dl_o\rho^p<t<s+\dl_o\rho^p.
\]
\end{proposition}
\begin{proof}
Assume $(y,s)=(0,0)$. By Proposition~\ref{Prop:5:1},
there exist $\dl=C^{-1}\al^{p+1}$ and $\eps=C^{-1}\al$, where $C>1$ depends only on the data, such that
\begin{equation}\label{Eq:7:1}
|[u(\cdot, t)>\eps M]\cap K_{\rho}|\ge\frac{\al}2|K_\rho|\quad\text{ for all }\quad0<t<\dl\rho^p.
\end{equation}

Now we set $Q^\prime=K_{2\rho}\times(0,\dl\rho^p]$ 
and $Q=K_{\rho}\times(\frac12\dl\rho^p,\dl\rho^p]$.
Apply \eqref{Eq:1:1} to $(u-M)_-$ with the pair of cylinders $Q\subset Q^\prime$ to obtain
\[
\iint_{Q}|D(u-M)_-|^p\,\dx\dt\le C\frac{M^p}{\dl\rho^p}|Q|.
\]
Under the change of variables
\[
x\to\frac{x}{\rho},\qquad t\to\frac{t}{\dl\rho^p},\qquad w=\frac{(u-M)_-}M,
\]
the above estimate reads
\begin{equation}\label{Eq:7:2}
\int^1_{\frac12}\int_{K_1}|Dw|^p\,\dx\dt\le\frac{C}{\al^{p+1}}.
\end{equation}
In order to use the lemma in \cite{DBGV}, we introduce $v=(1-w)/\eps$.
Then in terms of $v$ the measure information \eqref{Eq:7:1} reads
\begin{equation}\label{Eq:7:3}
|[v(\cdot,t)>1]\cap K_1|\ge\frac{\al}2|K_1|\quad\text{ for all }\quad\frac12<t<1.
\end{equation}
Combining \eqref{Eq:7:2} and \eqref{Eq:7:3}, there exists $\tau_1\in(\frac12,1]$
satisfying
\[
\int_{K_1}|Dv(\cdot,\tau_1)|\,\dx\le\frac{C}{\al^{p+1}},\qquad |[v(\cdot,\tau_1)>1]\cap K_1|\ge\frac{\al}2|K_1|.
\]
Now an application of the lemma in \cite{DBGV} yields that there exist $y_o\in K_1$ and
$\sig=C^{-1}\al^{4+\frac1p}$ for some absolute constant $C>1$, such that
\[
\left|\left[v(\cdot,\tau_1)>\frac12\right]\cap K_{\sig}(y_o)\right|\ge\frac12|K_\varep|.
\]
Returning to the original coordinates gives
\[
\left|\left[u(\cdot, t_1)>\frac12\eps M\right]\cap K_{\sig\rho}(x_o)\right|\ge\frac12|K_{\sig\rho}|
\]
for some $x_o\in K_\rho$ and $\frac12\dl\rho^p<t_1<\dl\rho^p$.
Using this measure information, we may apply Proposition~\ref{Prop:6:1} repeatedly
(choosing $\al=1/2$ in Proposition~\ref{Prop:6:1}) to obtain 
$\bar\eta,\,\bar\dl\in(0,1)$ depending only on the data,
such that for $n=1,2,\cdots$,
\[
u(\cdot, t)\ge\frac12\eps\bar{\eta}^n M\quad\text{ a.e. in }K_{2^n\sig\rho}(x_o)
\]
for all
\[
t_{n-1}+\frac12\bar{\dl}(2^n\sig\rho)^p<t<t_{n-1}+\bar{\dl}(2^n\sig\rho)^p\df{=}t_n.
\]
Finally we choose $n$ so large that $2^n\sig=3$, such that $K_{2\rho}\subset K_{2^n\sig\rho}(x_o)$.
At the same time, taking into consideration of the power-like dependence on $\al$
of $\eps$ and $\sig$, there exist $\eta_o\in(0,1)$ and $d>1$ depending only on the data, such that
\[
\eps\bar\eta^n=\eps2^{n\ln\bar{\eta}}=\eps\left(\frac{\sig}3\right)^{-\ln\bar\eta}=\eta_o\al^d.
\]
The time interval for such positivity is
\[
t_n-\frac12\bar{\dl}(3\rho)^p=t_n-\frac12\bar{\dl}(2^n\sig\rho)^p<t<t_n.
\]
We calculate $t_n$:
\[
t_n=t_1+\sum_{i=1}^{n-1} \bar{\dl}(2^{i}\sig\rho)^p=t_1+\frac{\bar{\dl}(2^{n}\sig\rho)^p-\bar{\dl}(2\sig\rho)^p}{2^p-1}
=t_1+\bar{\dl}\rho^p\frac{3^p-(2\sig)^p}{2^p-1}.
\]
With no loss of generality, we assume $\sig<1/4$. 
In this way, it is not hard to see there exist $\dl_o,\,\eta_o\in(0,1)$ depending
only on the data, such that
\[
u(\cdot, t)\ge\eta_o\al^d M\quad\text{ a.e. in }K_{2\rho}
\]
for all
\[
t_1+\frac12\dl_o\rho^p<t<t_1+\dl_o\rho^p.
\]
The qualitative location of $t_1\in(0,\rho^p)$ may be removed by
repeated applications of this conclusion.
The proof is then finished by properly redefining $\dl_o$.
\end{proof}
\subsubsection{Proof of Theorem~\ref{Thm:7:2}}
Assume $(y,s)=(0,0)$ and define
\[
I\df{=}\essinf_{K_{2\rho}\times(\frac12\dl_o\rho^p,\dl_o\rho^p]}u.
\]
We first estimate the $L^q$-norm of $u(\cdot,0)$ by its measure distribution:
\begin{equation}\label{Eq:7:4}
\begin{aligned}
\int_{K_{\rho}}u^q(\cdot,0)\,\dx&= q\int_0^\infty |[u(\cdot,0)>M]\cap K_\rho|M^{q-1}\,\d M\\
&\le q\int_I^\infty |[u(\cdot,0)>M]\cap K_\rho|M^{q-1}\,\d M+I^q|K_{\rho}|.
\end{aligned}
\end{equation}
By Theorem~\ref{Thm:7:2}, there exist $d>1$ and $\eta_o\in(0,1)$
depending only on the data, such that
\[
I\ge\eta_o M\left(\frac{|[u(\cdot,0)>M]\cap K_{\rho}|}{|K_\rho|}\right)^d.
\]
Thus we may estimate the first term on the right-hand side of \eqref{Eq:7:4} by
\[
q\int_I^\infty |[u(\cdot,0)>M]\cap K_\rho|M^{q-1}\,\d M\le\frac{qI^{\frac1d}}{\eta_o^{\frac1d}}|K_{\rho}|\int_I^\infty M^{q-\frac1d-1}\,\d M.
\]
Now we stipulate to take $q<1/d$, such that the improper integral on the right-hand side
converges. In such a way, the right-hand side of the above inequality is 
bounded above by $CI^q|K_\rho|$. Hence putting everything back in \eqref{Eq:7:4},
 we obtain the desired conclusion.

\end{document}

%% file: dibe_1009.tex
\newtheorem{proposition}{Proposition}[section]
\newtheorem{theorem}{Theorem}[section]
\newtheorem{lemma}{Lemma}[section]
\newtheorem{corollary}{Corollary}[section]
\newtheorem{remark}{Remark}[section]
\renewcommand{\thesection}{\arabic{section}}
\renewcommand{\theequation}{\thesection.\arabic{equation}}
\renewcommand{\thetheorem}{\thesection.\arabic{theorem}}
\numberwithin{equation}{section}
\numberwithin{theorem}{section}
\numberwithin{proposition}{section}
\numberwithin{lemma}{section}
\numberwithin{remark}{section}
\setcounter{secnumdepth}{3}
\newcommand{\cl}{\centerline}
\newcommand{\sms}{\smallskip}
\newcommand{\ms}{\medskip}
\newcommand{\bs}{\bigskip}
\newcommand{\noi}{\noindent}
\newcommand{\itl}[1]{\textit{#1}}
\newcommand{\blf}[1]{\textbf{#1}}
\newcommand{\dsty}{\displaystyle}
\newcommand{\txty}{\textstyle}
\newcommand{\ssty}{\scriptstyle}
\newcommand{\tty}{\texttt}


\newcommand\Par{\mathhexbox278\,}


\newcommand{\al}{\alpha}
\newcommand{\Al}{\Alpha}
\newcommand{\be}{\beta}
\newcommand{\Be}{\Beta}
\newcommand{\Gm}{\Gamma}
\newcommand{\gm}{\gamma}
\newcommand{\dl}{\delta}
\newcommand{\Dl}{\Delta}
\newcommand{\lm}{\lambda}
\newcommand{\Lm}{\Lambda}
\newcommand{\kp}{\kappa}
\newcommand{\varep}{\varepsilon}
\newcommand{\eps}{\epsilon}
\newcommand{\vp}{\varphi}
\newcommand{\sig}{\sigma}
\newcommand{\Sig}{\Sigma}
\newcommand{\om}{\omega}
\newcommand{\Om}{\Omega}
\newcommand{\uom}{\mbox{\boldmath$\omega$}}
\newcommand{\btau}{\mbox{\boldmath$\tau$}}
\newcommand{\bnu}{\mbox{\boldmath$\nu$}}
\newcommand{\up}{\upsilon}
\newcommand{\z}{\zeta}


\newcommand{\df}[1]{\buildrel\mbox{\small def}\over{#1}}
\newcommand{\op}[1]{\buildrel\mbox{\tiny o}\over{#1}}
\newcommand{\db}{\prime\prime}
\newcommand{\bsl}{\backslash}
\newcommand{\lb}{\lbrack\!\lbrack}
\newcommand{\rb}{\rbrack\!\rbrack}
\newcommand\la{\langle}
\newcommand\ra{\rangle}
\newcommand{\ev}{\equiv}
\newcommand{\nev}{\not\equiv}
\newcommand{\nn}{\mathbb{N}}
\newcommand{\qq}{\mathbb{Q}}
\newcommand{\zz}{\mathbb{Z}}
\newcommand{\rr}{\mathbb{R}}
\newcommand{\rn}{\rr^N}
\newcommand{\cc}{\mathbb{C}}
\newcommand{\id}{\mathbb{I}}
\newcommand{\bo}{\mathbb{O}}

\newcommand{\amsb}[1]{\mathbb{#1}}
\newcommand{\mcl}[1]{\mathcal{#1}}
\newcommand{\bl}[1]{\mathbf{#1}}
\newcommand{\ov}[1]{\overline{#1}}
\newcommand{\wt}[1]{\widetilde{#1}}
\newcommand{\wh}[1]{\widehat{#1}}

\newcommand{\llra}{\leftrightarrow}
\newcommand{\lra}{\longrightarrow}
\newcommand{\LLR}{\Longleftrightarrow}
\newcommand{\LRA}{\Longrightarrow}
\newcommand{\LLA}{\Longleftarrow}


\newcommand{\bbox}{\vrule height.6em width.6em 
depth0em} 
\newcommand{\os}{\vbox{\hrule \hbox{\vrule 
height.6em depth0pt 
\hskip.6em \vrule height.6em depth0em}
\hrule}} 


\newcommand{\dvg}{\operatorname{div}}
\newcommand{\curl}{\operatorname{curl}}
\newcommand{\supp}{\operatorname{supp}}
\newcommand{\essup}{\operatornamewithlimits{ess\,sup}}
\newcommand{\essinf}{\operatornamewithlimits{ess\,inf}}
\newcommand{\essosc}{\operatornamewithlimits{ess\,osc}}
\newcommand{\osc}{\operatornamewithlimits{osc}}
\newcommand{\sign}{\operatorname{sign}}
\newcommand{\loc}{\operatorname{loc}}
\newcommand{\diam}{\operatorname{diam}}
\newcommand{\dist}{\operatorname{dist}}
\newcommand{\card}{\operatorname{card}}
\newcommand{\meas}{\operatorname{meas}}
\newcommand{\spn}{\operatorname{span}}
\newcommand{\dtm}{\operatorname{det}}
%


\newcommand{\overlim}{\mathop{\overline{\lim}}\limits}
\newcommand{\underlim}{\mathop{\underline{\lim}}\limits}
\newcommand{\ttop}[2]{\genfrac{}{}{0pt}{}{#1}{#2}}
\newcommand{\bcu}{\mathop{\txty{\bigcup}}\limits}
\newcommand{\bca}{\mathop{\txty{\bigcap}}\limits}
\newcommand{\bsu}{\mathop{\txty{\sum}}\limits}
\newcommand{\pro}{\mathop{\txty{\prod}}\limits}


\newcommand{\pl}{\partial}
\newcommand{\ptt}{\frac{\pl}{\pl t}}
\newcommand{\ppx}{\frac\pl{\pl x}}
\newcommand{\dds}{\frac d{ds}}
\newcommand{\ddt}{\frac d{dt}}

\newcommand{\intl}{\int\limits}
\newcommand{\iintl}{\iint\limits}
\def\Xint#1{\mathchoice
    {\XXint\displaystyle\textstyle{#1}}%
    {\XXint\textstyle\scriptstyle{#1}}%
    {\XXint\scriptstyle\scriptscriptstyle{#1}}%
    {\XXint\scriptscriptstyle\scriptscriptstyle{#1}}%
    \!\int}
\def\XXint#1#2#3{\setbox0=\hbox{$#1{#2#3}{\int}$}
    \vcenter{\hbox{$#2#3$}}\kern-0.5\wd0}
\def\bint{\Xint-}
\def\dashint{\Xint{\raise4pt\hbox to7pt{\hrulefill}}}
\def\dashiint{\bint\kern-0.15cm\bint}

\newcommand{\ovl}[3]{\int_{#1}^{#2}\kern-#3pt\raise4pt\hbox to7pt{\hrulefill}\ }

\newcommand{\ovll}[3]{\intl_{#1}^{#2}\kern-#3pt\raise4pt\hbox to7pt{\hrulefill}\ }

\newcommand{\tvl}[2]{\iint_{#1}\kern-#2pt\raise4pt\hbox to7pt{\hrulefill}\ }



\newcommand{\omt}{\Om_T}
\newcommand{\plo}{\partial\Omega}
\newcommand{\ovo}{\bar{\Om} }

%
\newcommand{\ci}[1]{C^\infty\!\left({#1}\right)}
\newcommand{\cio}[1]{C_o^\infty\!\left({#1}\right)}
\newcommand{\lloc}[1]{L_{\loc}\!\left({#1}\right)}
\newcommand{\xy}{|x-y|}


\newcommand{\intom}{\intl_{\Om}}
\newcommand{\intbo}{\intl_{\plo}}
\newcommand{\inom}{\int_{\Om}}
\newcommand{\inbo}{\int_{\plo}}
\newcommand{\intrn}{\intl_{\rn}}


\newcommand{\bye}{\end{document}}



%
%
%

%
%
%
%
%

%
%